\documentclass[12pt]{amsart}
\usepackage[margin=.85in]{geometry}
\usepackage{amsmath,amsthm,amssymb}
\usepackage{comment}
\usepackage{hyperref}
\usepackage{color}
\newcommand{\stkout}[1]{\ifmmode\text{\sout{\ensuremath{#1}}}\else\sout{#1}\fi}
\usepackage{graphicx}

\theoremstyle{plain}
\newtheorem{thm}{Theorem}[section]
\newtheorem{cor}[thm]{Corollary}
\newtheorem{lem}[thm]{Lemma}
\newtheorem{prop}[thm]{Proposition}
\newtheorem{open}[thm]{Open Question}

\theoremstyle{definition}
\newtheorem{defn}[thm]{Definition}

\theoremstyle{remark}
\newtheorem{rem}[thm]{Remark}
\newtheorem{ex}[thm]{Example}
\numberwithin{equation}{section}
\newcommand{\R}{\mathbb R}
\newcommand{\N}{\mathbb N}
\newcommand{\Z}{\mathbb Z}
\newcommand{\cH}{{\mathcal H}}
\newcommand{\CF}{{\mathcal F}}
\newcommand{\cF}{{\mathcal F}}
\newcommand{\cR}{{\mathcal R}}
\newcommand{\cM}{{\mathcal M}}
\newcommand{\cB}{{\mathcal B}}
\newcommand{\CW}{{\mathcal B}}
\newcommand{\cD}{{\mathcal D}}
\newcommand{\CD}{{\mathcal D}}
\newcommand{\cC}{{\mathcal C}}
\newcommand{\la}{\langle}
\newcommand{\bn}{{\bf n}}
\newcommand{\ra}{\rangle}
\newcommand{\re}{{\rm Re}}
\newcommand{\be}[1]{\begin{equation}\label{#1}}
\newcommand{\ee}{\end{equation}}
\newcommand{\ba}{\begin{align*}}
\newcommand{\ea}{\end{align*}}
\newcommand{\bigo}[1]{O\left( #1 \right)}
\newcommand{\smallo}[1]{o\left( #1 \right)}

\newcommand{\de}{\, \mathrm{d}}
\begin{document}
\title{Eigenvalue bounds of mixed Steklov problems\\
}%
\author{Asma Hassannezhad}%
\address
{Asma Hassannezhad: University of Bristol,
School of Mathematics,
University Walk,
Bristol BS8 1TW, UK}
\email{asma.hassannezhad@bristol.ac.uk}
\author{Ari Laptev}
\address
{Ari Laptev: Imperial College London, Department of Mathematics, 180 Queen's Gate, 
London SW7 2AZ, UK}
\email{a.laptev@imperial.ac.uk}
\address{Francesco Ferrulli: Imperial College London, Department of Mathematics, 180 Queen's Gate, 
London SW7 2AZ, UK}
\email{f.ferrulli14@imperial.ac.uk}
\address{Jean Lagac\'e: University College London, Department of Mathematics, Gower Street, London, WC1E 6BT, UK}
\email{j.lagace@ucl.ac.uk}

% ----------------------------------------------------------------
%\begin{abstract}
%\end{abstract}
%\maketitle
\setcounter{page}{1}
% ----------------------------------------------------------------
\maketitle
\centerline{\small(with an appendix by F. Ferrulli and J. Lagac\'e)}

\begin{abstract} We study bounds on the Riesz means of the mixed Steklov--Neumann and Steklov--Dirichlet eigenvalue problem on a  bounded domain $\Omega$ in $\R^n$. The Steklov--Neumann eigenvalue problem is also called the sloshing problem. We obtain two-term asymptotically sharp lower bounds on the Riesz means  of the sloshing problem and also provide an asymptotically sharp upper bound for the Riesz means   of mixed Steklov--Dirichlet problem. The proof of our results for the sloshing problem uses the average variational principle and monotonicity of sloshing eigenvalues. In the case of Steklov-Dirichlet eigenvalue problem, the proof is based on  a well--known bound on the Riesz means of the Dirichlet fractional Laplacian, and an inequality between the Dirichlet and Navier fractional Laplacian.  The two-term asymptotic results for the Riesz means of mixed Steklov eigenvalue problems are discussed in the appendix which in particular show the asymptotic sharpness of the bounds we obtain.    \end{abstract}
\section{Introduction} 
Let $\Omega$ be a bounded domain  in $\R^{n}$ with Lipschitz and piecewise smooth boundary $\partial\Omega$. We assume that 
\begin{equation}\label{hypothesis}
\partial \Omega=\cF\cup \cB,\quad \text{where $\cF\subset \{x_n=0\}$, and $\cB\subset\{x_n<0\}$}.
\end{equation}  
Throughout the paper, we refer to $\cF$ as a subset of $\R^{n-1}\times\{0\}$ and as a subset of $\R^{n-1}$ interchangeably. Consider the following eigenvalue problem 
\be{slosh}\begin{cases}
\Delta f=0,\,& \hbox{in $\Omega$},\\
\frac{\partial f}{\partial \bn}=0,\,&\hbox{on $\cB$},\\
\frac{\partial f}{\partial x_n}=\nu f,\,&\hbox{on $\cF$},
\end{cases}\ee
where $\bn$ is the unite outward normal vector along $\partial \Omega$, and $\frac{\partial f}{\partial \bn}$ is the derivative of $f$ in the direction of $\bn$.  The above mixed Steklov--Neumann eigenvalue problem is also called the sloshing problem. It is known
 that it has a discrete set of eigenvalues  (see for example~\cite[Chapter III]{ban80})
$$0=\nu_1\le\nu_2\le\nu_3\le\cdots\nearrow\infty$$
and each eigenvalue has a finite multiplicity. The corresponding eigenfunctions $\{\varphi_j\}_{j=1}^\infty$ restricted to the free surface $\cF$ form a basis for $L^2(\cF)$. The eigenvalues of the sloshing problem can be considered as the eigenvalues of the Dirichlet--to--Neumann map 
$$
\cD_N:L^2(\cF)\to L^2(\cF),
$$
$$
f\mapsto\frac{\partial \tilde f}{\partial \bn},
$$
where $\tilde f$ is the harmonic extension of $f$ to $\Omega$ satisfying the Neumann boundary condition on $\cB$.\\
The sloshing problem naturally appears in the study of the sloshing liquid, where the sloshing frequency is proportional to $\sqrt{\nu_j}$ and its study has a long history.   For a short historical note we refer to~\cite{history}, and for more recent developments on the subject to~\cite{Petal10,LPPS} and the references therein.\\ 
The focus of our  study  is to find  sharp semiclassical lower/upper bounds for the Riesz means of eigenvalues of the mixed Steklov problems \eqref{slosh} and \eqref{dirstek} (see below).\\
The Riesz mean $R_\gamma(z)$ of  order $\gamma>0$ is defined as
\[
R_\gamma(z):=\sum_j(z-\nu_j)^\gamma_+,\qquad z>0,
\]
where $(z-\nu)_+:=\max\{0,z-\nu\}$. We may also denote it by $R^\Omega_\gamma(z,\cD_N)$ to identify the domain and the operator under consideration.\\
When $\gamma\to0$, it approaches the counting function 
$$
N(z):=\sum_{\nu_j< z}1=\sup\{k:\nu_k< z\}
$$
and by convention we denote $R_0(z):=N(z)$. The asymptotics of the counting function $N(z)$ for the eigenvalue problem \eqref{slosh} is given by (see for example~\cite{San55})
\[
N(z)\sim \frac{\omega_{n-1}}{(2\pi)^{n-1}}|\cF|z^{n-1},\qquad z\nearrow \infty,
\]
where $\omega_{n-1}=\frac{\pi^{\frac{n-1}{2}}}{\Gamma(\frac{n+1}{2})}$ is the volume of a unit ball in $\R^{n-1}$, and $|\cF|$ denote the $(n-1)$-Euclidean volume of $\cF$. Using the Riesz iteration, i.e. the following identities 
 \begin{equation}\label{iteration}
R_{\gamma+ \rho }(z)= \dfrac{\Gamma(\gamma+\rho+1)}
{\Gamma(\gamma+1) \ \Gamma(\rho)} \int_{0}^{\infty}
\left(z-t\right)_{+}^{\rho-1} R_{\gamma}(t) dt
\end{equation}
and
  \[ 
 R_{\gamma}(z)= \gamma \int_{0}^{\infty}
\left(z-t\right)_{+}^{\gamma-1} R_{0}(t) dt=\gamma \int_{0}^{z}
\left(z-t\right)^{\gamma-1} R_{0}(t) dt,
\]
 we can immediately get the asymptotics behaviour of $R_\gamma(z)$
\begin{equation}\label{rieszasym}
R_\gamma(z)\sim C_{n,\gamma}|\cF|z^{n+\gamma-1},\qquad z\nearrow\infty,
\end{equation}
where 
$$
C_{n,\gamma}:=\frac{1}{(4\pi)^{\frac{n-1}{2}}}\frac{\Gamma(\gamma+1)\Gamma(n)}{\Gamma(\frac{n+1}{2})\Gamma(n+\gamma)}.
$$
For  basic facts on the Riesz means, we refer to~\cite{CM52,HH11}.
Sharp semiclassical bounds on the Riesz means of Dirichlet and Neumann eigenvalues of the Laplacian were studied in numerous work, see for example~\cite{Lap97,GLW11,HS16,LY83}. Recently, a sharp semiclassical bound for the Riesz means $R_\gamma(z)$, $\gamma\ge2$, of Steklov eigenvalues was obtained in~\cite{PS16}. However,  such sharp semiclassical bounds for Riesz means  of the mixed Steklov problem are unknown. \\
 We  also consider the mixed Steklov--Dirichlet eigenvalue problem:   
\be{dirstek}\begin{cases}
\Delta f=0,\,& \hbox{in $\Omega$},\\
 f=0,\,&\hbox{on $\cB$},\\
\frac{\partial f}{\partial x_n}=\eta f,\,&\hbox{on $\cF$},
\end{cases}\ee
where instead of the Neumann boundary condition, the Dirichlet boundary condition is imposed  on $\cB$. For a physical interpretation of this problem, see for example~\cite{LPPS}.  The eigenvalues of the Steklov--Dirichlet problem can be also considered as the the eigenvalues of the Dirichlet--to--Neumann map 
$$\cD_D:L^2(\cF)\to L^2(\cF)
$$
$$
f\mapsto\frac{\partial \tilde f}{\partial \bn},
$$
where $\tilde f$ is the harmonic extension of $f$ to $\Omega$ satisfying the Dirichlet boundary condition on $\cB$. We  denote the Riesz means of eigenvalues of problem \eqref{dirstek} by $R^\Omega_\gamma(z,\cD_D)$, and when there is no confusion  by $R^\Omega_\gamma(z)$ or $R_\gamma(z)$. \\
In the following subsections, we state our main results on asymptotically sharp bounds on $R^\Omega_\gamma(z,\cD_N)$ and $R^\Omega_\gamma(z,\cD_D)$. {As a consequence, we also get asymptotically sharp bounds on the sum of first $k$ eigenvalues of the mixed Steklov problem. }
\subsection{Steklov--Neumann eigenvalue problem}\label{intro1}  Our first result gives a two--term asymptotically sharp lower bound on $R_\gamma(z,\cD_N)$ in dimension two. 
\begin{thm}\label{triangle} Let $\tilde\Omega$ be a bounded domain in $\R^2$ with $\partial\tilde\Omega=\cF\cup\tilde\cB$ as in \eqref{hypothesis}. We assume that  $\cF$ is connected and there exists $\delta>0$ such that $\tilde\cB$ meets $\cF$  in two line segments in a $\delta-$neighbourhood of  corner points\footnote{The points in the intersection $\bar\cF$ and $\bar\cB$ are called the corner points.} $P$ and $Q$ (as shown in Figure~\ref{pic1}) with angles $\alpha,\beta\in(0,\pi)$. We denote the complement of these two line segments in $\tilde\cB$  by $\tilde\cB_c$.  Then for every $\gamma\ge1$ and every  $z>0$ there exists a constant  $c=c(z,\gamma, \delta,\alpha,\beta,|\tilde\cB_c|)$  depending on $z,\gamma, \delta,\alpha, \beta$, and $|\tilde\cB_c|$ such that for any $\Omega\subset\tilde\Omega$ with $\partial\Omega=\cF\cup\cB$ the Riesz mean $R^\Omega_\gamma(z,\cD_N)$   satisfies the following inequality. 
\begin{equation}\label{1.4}
R^\Omega_\gamma(z,\cD_N)\ge C_{2,\gamma}|\cF|z^{\gamma+1}+\frac{1}{2\pi}\left(\frac{1}{\tan(\alpha)}+\frac{1}{\tan(\beta)}\right)z^\gamma+c,
\end{equation}
where $C_{2,\gamma}=\frac{1}{\pi(\gamma+1)}$. Moreover, $c=O(z^{\gamma-1})$ as $z\to\infty$.  Here our convention is that $\frac{1}{\tan(\pi/2)}=0$.
\end{thm}
\begin{figure}
\centering{
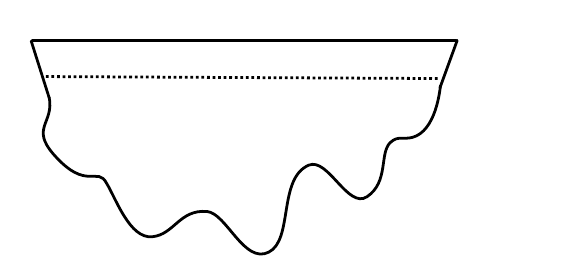
\caption{}\label{pic1}}
\end{figure}

\begin{rem}
Note that  for a fixed $z$,  when $\delta$ tends to $0$, constant $c$ in inequality \eqref{1.4} may tend to~$-\infty$. However, there exists a contant $z_0$  depending on $\gamma, \delta,\alpha, \beta$, and $|\tilde\cB_c|$ with $(\alpha,\beta)\ne(\frac{\pi}{2},\frac{\pi}{2})$ such that for every $z\ge z_0$ and $\gamma\ge1$, and for any  domain $\Omega$ satisfying the assumption of Theorem~\ref{triangle},
 we have,
\[
R^\Omega_\gamma(z,\cD_N)\ge C_{2,\gamma}|\cF|z^{\gamma+1}+\frac{1}{4\pi}\left(\frac{1}{\tan(\alpha)}+\frac{1}{\tan(\beta)}\right)z^\gamma.
\]
\end{rem}
When $\tilde\Omega$ is a triangle/trapezoid, then for $\delta$ equal to the height of the triangle/trapezoid, we have $\tilde\cB_c=\emptyset$, and constant $c$ in $\eqref{1.4}$ is equal to
$$c=\frac{1}{4\pi\delta}\left(\frac{1}{\tan(\alpha)}+\frac{1}{\tan(\beta)}\right)(1-e^{-2\delta z}).$$
 The coefficient  of $z^\gamma$ in  \eqref{1.4} is zero when $(\alpha,\beta)=(\frac{\pi}{2},\frac{\pi}{2})$.
 When $\Omega$ is a subset of the infinite strip $\cF\times(-\infty,0)$, we say that $\Omega$ satisfies the so-called \textit{John condition}, see \cite{Petal10}. We show that when $\Omega$ satisfies the John condition then we {get a uniform lower bound with}  $\frac{1}{2}z^\gamma$ in the second term:
\begin{equation}\label{recintro}R^\Omega_\gamma(z,\cD_N)\ge\frac{|\cF|}{\pi(\gamma+1)}z^{\gamma+1}+\frac{1}{2}z^{\gamma}.\end{equation}
We remark that inequality \eqref{recintro} is not a consequence of Theorem \ref{triangle}, see Proposition~\ref{correctangle}.\\

  In Theorem \ref{triangle},  when $\cB$ and $\tilde\cB$ are tangent to each other at points $P$ and $Q$, then the coefficient of $z^\gamma$ depends only  on the interior angles between $\cF$ and $\cB$. \\
  Recently,
  Levitin, Parnovski, Polterovich and Sher in~\cite[Theorem 1.2.2]{LPPS}  showed  that when $\alpha,\beta\in(0,\frac{\pi}{2})$ are the interior angles between $\cF$ and $\cB$, and $k\nearrow \infty$, the following asymptotic expansion holds 
\begin{equation}\label{2asym0}
\nu_k|\cF|= \pi k-\frac{\pi}{2}-\frac{\pi^2}{8}\left(\frac{1}{\alpha}+\frac{1}{\beta}\right)+o(1).
\end{equation}
Their result in particular proves  Fox-Kuttler's conjecture in 1983~\cite{FK83}. From \eqref{2asym0}, one can deduce $N(z)= \frac{1}{\pi}|\cF|z+O(1)$, as $z\nearrow\infty$ (see also \cite[Corollary 1.6.1]{LPPS} for a related result).   Ferrulli  and Lagac\'e show (see the appendix) that the following two-term asymptotic for $R_\gamma(z)$, $\gamma>0$ is a consequence of \eqref{2asym0}.   
\begin{equation}\label{FL}R_\gamma(z,\cD_N)=
C_{2,\gamma}|\cF|z^{\gamma+1}+\frac{\pi}{8}\left(\frac{1}{\alpha}+\frac{1}{\beta}\right)z^\gamma+o(z^\gamma),
\end{equation}
where $\alpha,\beta\in(0,\frac{\pi}{2})$ are the interior angles between $\cF$ and $\cB$. The above asymptotic holds when one or both  angles take the value $\frac{\pi}{2}$ provided that  $\Omega$ satisfies the \textit{local John condition} (we refer to the appendix for the definition).   We observe that the coefficient of $z^\gamma$ in the second term of inequality \eqref{1.4} depends on the same quantities appearing in the coefficient of $z^\gamma$ in the two--term asymptotic expansion~\eqref{2asym0}. In particular, for domains satisfying the John condition, the second term of \eqref{recintro} is also asymptotically sharp.  

In higher dimensions, we  obtain a general formula for a two--term lower bound on $R_\gamma(z,\cD_N)$, $\gamma\ge1$, see Theorem~\ref{main}. Here, we only mention a corollary of  Theorem~\ref{main}.  
\begin{thm}\label{maincor} Assume that $\Omega$ with $\partial\Omega=\cF\cup\cB$ satisfied the John condition, i.e. is  a subset of $\cF\times(-\infty,0)$. Then 
\begin{equation} 
R_1^\Omega(z)\ge C_{n,1}|\cF|z^{n}+\frac{(n-1)\omega_{n-1}}{(2\pi)^{n-1}}\frac{|\cF|}{(2h_\Omega)^n}\left(\Gamma(n)-\Gamma(n,2h_\Omega z)\right),
\end{equation}
where $h_\Omega$ is the depth of~$\Omega$.
\end{thm} 
We recall the definition of the incomplete $\Gamma$-function $\Gamma(n,x)$:
$$
\Gamma(n,x):=(n-1)!e^{-x}\sum_{k=0}^{n-1}\frac{x^k}{k!}.
$$
In particular, notice 
$$
\Gamma(n)-\Gamma(n,x)>0,\qquad \forall x>0, \quad \forall n\in\N.
$$
One can apply the Riesz iteration in the inequalities above and get a lower bound on $R_\gamma^\Omega(z,\cD_N)$.\\
The leading term in \eqref{maincor} is asymptotically sharp. To the best of our knowledge  no two-term asymptotic expansion is known in higher dimensions in order to compare our results with. 

The  proofs of the results above are based on using  the averaged variational principal introduced in~\cite{EHIS} and monotonicity results for eigenvalues of \eqref{slosh}  studied in~\cite{Petal10}. \\

We now give an asymptotically sharp bounds for the sum of first $k$th eigenvalues of the sloshing problem. 
Kr\"oger~\cite{Kr92} obtained  a sharp upper bound for the sum of eigenvalues of the Laplacian with Neumann boundary condition. His result was recently sharpen by Harrell and Stubbe in~\cite{HS16}. Adapting the  argument in~\cite{HS16} to the sloshing eigenvalue problem, we  obtain a counterpart of the Kr\"oger--Harrell--Stubbe inequality for the sum of eigenvalues of the sloshing problem. 
 \begin{thm}\label{kroger}
Under the assumption of Theorem~\ref{maincor}, we have 
\begin{equation}\label{kroger1}
\frac{1}{k}\sum_{j=1}^k\nu_j\le \frac{n-1}{n}\left(W_{n,k}-\frac{1}{W_{n,k}}(\nu_{k+1}-W_{n,k})^2\right),
\end{equation}
where 
\[
W_{n,k}:=2\pi\omega_{n-1}^{-\frac{1}{n-1}}\left(\frac{k}{|\cF|}\right)^{\frac{1}{n-1}}.
\]
Inequality \eqref{kroger1} in particular gives a two-sided asymptotically sharp  bound for an individual eigenvalue 
\begin{equation}
W_{n,k}(1-\sqrt{1-S_k})\le \nu_{k+1}\le W_{n,k}(1+\sqrt{1-S_k}),
\end{equation}
where 
$$
S_k:=\frac{n}{n-1}\frac{\sum_{j=1}^k\nu_j}{kW_{n,k}}.
$$
\end{thm}
%---------------------------------
\subsection{Steklov--Dirichlet eigenvalue problem}\label{intro2} In this section, we state our result on asymptoticly sharp upper bounds on $R_\gamma^\Omega(z,\cD_D)$, $\gamma\ge1$.
 The same asymptotic \eqref{rieszasym} remains true for the mixed Steklov--Dirichlet problem~\eqref{dirstek}.  
 \begin{equation}
 R_\gamma(z)\sim C_{n,\gamma}|\cF|z^{n+\gamma-1},\qquad z\nearrow\infty.
 \end{equation}
 We obtain an asymptotically sharp upper bound for $R_\gamma^\Omega(z,\cD_D)$: 
 %....
\begin{thm}\label{adirstek}Let $\Omega$ be a bounded domain and subset of an infinite cylinder  $\cF\times[-\infty,0]$, where $\partial\Omega=\cF\cup\cB$. Here $\cF$ is the free part of the boundary. Then for every $\gamma\ge1$ and $z>0$ we have
\begin{equation}\label{1.6} R^\Omega_\gamma(z,\cD_D)=\sum_{j}(z-\eta_j)^\gamma_+\le C_{n,\gamma}|\cF|z^{n+\gamma-1}.
\end{equation}
 \end{thm}
 %---
 For the proof of this theorem, we use the relationship between eigenvalues of the fractional Laplacian with different type of boundary conditions studied in~\cite{MN14,MN16,MN16b} together with the result of Laptev~\cite{Lap97} on upper bounds of the Riesz means for the Dirichlet fractional Laplacian. \\
 %-----
 As pointed out in \cite{LW00} and \cite[Page 8]{LW00b},   an asymptotically sharp upper bound for $R_1(z)$ leads to an asymptotically sharp lower bound for the sum of first $k$th eigenvalues (and vice versa) by  using  the Legendre transform. Thus,  we get the following bound on the sum of firth $k$th Steklov-Dirichlet eigenvalues by applying the Legendre transform  to \eqref{1.6},:
 \begin{cor}
 Under the assumption of Theorem \ref{adirstek}, the following inequality holds.
 \begin{equation}
\frac{1}{k}\sum_{j=1}^k\eta_j\ge  \frac{n-1}{n}W_{n,k}=2\pi\left(\frac{n-1}{n}\right)\omega_{n-1}^{-\frac{1}{n-1}}\left(\frac{k}{|\cF|}\right)^{\frac{1}{n-1}}.
 \end{equation}
 \end{cor}
 
 %-----
  As in \eqref{FL},  Ferrulli and Lagac\'e also obtain the following two-term asymptotic  for  $R_\gamma(z,\cD_N)$. We refer to the appendix for more details. \begin{equation}\label{FL2}R_\gamma(z,\cD_D)=
C_{2,\gamma}|\cF|z^{\gamma+1}-\frac{\pi}{8}\left(\frac{1}{\alpha}+\frac{1}{\beta}\right)z^\gamma+o(z^\gamma),
\end{equation} where $\alpha,\beta\in(0,\frac{\pi}{2})$ are the interior angles between $\cF$ and $\cB$. This asymptotic holds when $\alpha$ or $\beta$ is equal to $\frac{\pi}{2}$ provided that $\Omega$ satisfies the local John condition. We show that when  $\Omega$ satisfies the John condition, we can get two-term asymptotically sharp upper  bound for  $R^\Omega_1(z,\cD_D)$.  %----
\begin{thm}\label{recdir}
Let $\Omega\subset\R^2$ satisfies the John condition. Then the Riesz mean of eigenvalues $\eta_j$ of problem \eqref{dirstek} satisfies
 \begin{equation}\label{eqrecdir}R^\Omega_1(z,\cD_D)=\sum_{j}(z-\eta_j)_+\le\frac{|\cF|}{2\pi}z^2-\frac{1}{2}z+\frac{\pi}{2|\cF|}.\end{equation}
\end{thm}
%---
One can apply the Riesz iteration to find bounds on $R_\gamma(z,\cD_D)$.
\begin{open}It is an intriguing question if we can get a two--term upper bound with a negative second term depending only on $\alpha$ and $\beta$.  One can ask if there exist a positive constant $C(\alpha,\beta)$  such that
 $$
 \sum_{j}(z-\eta_j)^\gamma_+\le C_{n,\gamma}|\cF|z^{n+\gamma-1}-C(\alpha,\beta)z^{n+\gamma-2}.
 $$
 \end{open}

The paper is organised as follows. In section~\ref{2}, we prove the main results  on bounds on the Riesz means of eigenvalues of problem \eqref{slosh}, and on Kr\"oger-Harrell-Stubbe's type inequality for eigenvalues of the sloshing problem. We also consider cases in which we can get more explicit lower bounds. In section~\ref{4}, we study the upper bounds on the Riesz means of eigenvalues of the mixed Steklov-Dirichlet problem \eqref{dirstek}. 
\subsection*{Acknowledgements} The authors would like to thank M. van den Berg and I. Polterovich for interesting discussion.  The main part of this work was completed when the first named author was an EPDI postdoctoral fellow in the Mittag-Leffler Institute. She is grateful to this institute for the support and for providing an excellent working condition. Funding for this research was partially provided by the grant of the Russian Federation Government to support research under the supervision of a leading scientist at the Siberian Federal University, 14.Y26.31.0006.

%---------------
\section{Slosing eigenvalue problem} \label{2}
In this section, we prove the results of Section~\ref{intro1} of the introduction. 
%\subsection{Bounds on the Riesz means}
We first recall the variational characterisation of the eigenvalues of the mixed Steklov-Neumann problem. Let 
$\{\varphi_j\}_{j=1}^\infty$ be a sequence of eigenfunctions associate with $\{\nu_j\}_{j=1}^\infty$. The $k$-th eigenvalue $\nu_k$ of  problem \eqref{slosh} can be characterised by 
\be{varch}
\nu_k=\inf_{0\ne f\in H_k}R(f),
\ee
where $H_k=\left\{g\in H^1(\Omega): \int_\cF g\varphi_j ds=0, j=1,\ldots,k-1\right\}$, and
\[
R(f):=\frac{\int_\Omega|\nabla f|^2dx}{\int_{\cF}|f|^2 ds}.
\] 
Let $\cH(\Omega)$ denotes the space of harmonic functions on $\Omega$. In \eqref{varch} one can replace $H_k$ by $\cH_k:=\{f\in \cH(\Omega): \int_\cF g\varphi_j ds=0, j=1,\ldots,k-1\}$.
We recall the so-called averaged variational principle introduced in~\cite{EHIS}. 
Let $f\in H^1(\Omega)$ and $z\in(\nu_{k-1},\nu_k]$. We choose   $\{\varphi_j\}$ so that their  restriction to $\cF$ forms an orthonormal basis for $L^2(\cF)$.  Thus, by \eqref{varch} we have
\[z\le R\left(f-\sum_{j=1}^{k-1}\la\varphi_j,f\ra_\cF\varphi_j\right)=\frac{\int_\Omega|\nabla f|^2\,dx-\sum_{j=1}^{k-1}\nu_j|\la\varphi_j,f\ra|^2_\cF}{\int_\cF|f|^2\,ds-\sum_{j=1}^{k-1}|\la\varphi_j,f\ra|^2_\cF},\]
where $\la f, g\ra_\cF:=\int_\cF f\bar g\,ds$.
Therefore,
\be{mminimax}\sum_j(z-\nu_j)_+\left|\la \varphi_j,f\ra\right|^2_\cF\ge z\int_\cF |f|^2ds-\int_\Omega|\nabla f|^2dx.\ee
 If moreover $f\in \cH(\Omega)$, then applying the Green formula we get
\be{mvar}\sum_j(z-\nu_j)_+|\la \varphi_j,f\ra|^2_\cF\ge z\int_\cF |f|^2ds-\re\int_{\partial\Omega} \frac{\partial f}{\partial\bn} \bar f\,ds.\ee
We summarise the discussion above in the following lemma which is called the averaged variational principle. This  is an special case of a more general statement in~\cite[Lemma 1.5]{EHIS}.
\begin{lem}[averaged variational principle]\label{lemav} Let $f_\xi\in\cH(\Omega)$ be a family of harmonic functions where $\xi$ varies over a measure space $(M, \cM,\mu)$, with $\sigma$-algebra $\cM$. Let $M_0$ be a measurable subset of $M$. Then for any $z\in\R_+$ we have
\be{mvarg}\sum_j(z-\nu_j)_+\int_{M}|\la \varphi_j,f_\xi\ra|^2_\cF d\mu\ge z\int_{M_0}\int_\cF |f_\xi|^2ds\,d\mu-\int_{M_0}\re\int_{\partial\Omega} \frac{\partial f_\xi}{\partial\bn} \bar f_\xi\,ds\,d\mu.\ee
\end{lem}
Another key lemma we need is the monotonicity results for the mixed Steklov-Neumann eigenvalues:
\begin{lem}{\cite[Proposition 3.2.1]{Petal10}}\label{monotoni}
Let $\Omega$ and $\tilde\Omega$ be subdomains of $\R^n$ whose boundaries $\partial\Omega=\cF\cup\cB$ and $\partial\tilde\Omega=\tilde\cF\cup\tilde\cB$ are as described in \eqref{hypothesis}.  Let $\Omega$ be a proper subset of $\tilde\Omega$ and $\tilde\cF=\cF$. %We denote eigenvalues  of problem \eqref{slosh} on $\Omega$ and $\tilde\Omega$ by $\nu_k$ and $\tilde\nu_k$ respectively. 
Then the following inequality holds.
\[\nu_{k}(\Omega)<\nu_{k}(\tilde\Omega),\qquad\forall k\ge2 .\]
In particular,
$$
R^\Omega_\gamma(z)=\sum_j(z-\nu_j(\Omega))_+\ge\sum_j(z-\nu_j(\tilde\Omega))_+= R^{\tilde\Omega}_\gamma(z).
$$
\end{lem} 
We can now state a general form of the results mentioned in the introduction.
\begin{thm}\label{main} Let $\Omega$ be a bounded domain of $\R^n$ and $\partial \Omega=\cF\cup\cB$ as described in \eqref{hypothesis}. The Riesz means $R_\gamma(z)$, $\gamma\ge1$,  of the eigenvalues of the mixed Steklov-Neumann problem~\eqref{slosh} satisfy the following inequality. 
% adding some motivation and physical background on sloshing problem. 
\be{rieszinq}
R_\gamma(z)\ge C_{n,\gamma}|\cF|z^{n+\gamma-1}+ A_{n,\gamma}(z),
\ee
where 

\[
A_{n,\gamma}(z)=-\frac{(n-1)\gamma(\gamma-1)}{(4\pi)^{\frac{n-1}{2}}\Gamma(\frac{n+1}{2})}\int_0^z(z-\eta)^{\gamma-2}\int_{0}^\eta \int_{\cB} \la\bn(x),e_n\ra e^{2x_nr}r^{n-1}\,ds(x)drd\eta.
\]
Here $ds(x)$ is the volume element on $\cB$ and $\la\cdot,\cdot\ra$ is the inner product in $\R^n$.
\end{thm}
Note that \begin{equation}\label{a1gamma}A_{n,1}(z):=-\frac{(n-1)\omega_{n-1}}{(2\pi)^{n-1}}\int_{0}^z \int_{\cB} \la\bn(x),e_n\ra e^{2x_nr}r^{n-1}\,ds(x)dr.\end{equation}

It is clear that when $ \la\bn(x),e_n\ra\le 0$ for all $x\in \cB$, then $A_{n,\gamma}(z)$ is positive. Bellow, we discuss situations in which we have more explicit estimates on $A_{n,1}$. Estimates on $A_{n,\gamma}$ following by using the Riesz itteration.
\begin{cor}\label{maincor2} Assume that there exists $\delta>0$ such that 
$$
\delta\le\min\{|x_n|: x=(x',x_n)\in\cB,\;\la\bn(x),e_n\ra> 0\}.
$$
Then 
\begin{multline}\label{ah1}A_{n,1}(z)\ge\frac{(n-1)\omega_{n-1}}{(2\pi)^{n-1}}\left(\left(\int_{\cB^-} |\la\bn(x),e_n\ra |ds\right)\frac{1}{(2h_\Omega)^n}\left(\Gamma(n)-\Gamma(n,2h_\Omega z)\right)\right.\\-\left.\left(\int_{\cB^+} \la\bn(x),e_n\ra ds\right)\frac{1}{(2\delta)^n}\left(\Gamma(n)-\Gamma(n,2\delta z)\right)\right), 
\end{multline}
where $\cB^{+}:=\{x\in \cB: \la\bn(x),e_n\ra>0 \}$, and $\cB^{-}:=\{x\in \cB: \la\bn(x),e_n\ra\le0 \}$.
In particular, when $\cB^+=\emptyset$, we have
\begin{equation}\label{ah2}A_{n,1}(z)\ge\frac{(n-1)\omega_{n-1}}{(2\pi)^{n-1}}\left(\left(\int_{\cB^-} \la\bn(x),e_n\ra ds\right)\frac{1}{(2h_\Omega)^n}\left(\Gamma(n)-\Gamma(n,2h_\Omega z)\right)\right). \end{equation}
\end{cor} 
\begin{proof} By Theorem~\ref{main} we have
\begin{multline*}A_{n,1}(z)=\frac{(n-1)\omega_{n-1}}{(2\pi)^{n-1}}\left(\int_{0}^z \int_{\cB^-} |\la\bn(x),e_n\ra| e^{2x_nr}r^{n-1}\,ds(x)dr\right.\\-\left.\int_{0}^z \int_{\cB^+} \la\bn(x),e_n\ra e^{2x_nr}r^{n-1}\,ds(x)dr\right). \end{multline*}
Since $x_n<0$, the function $e^{2x_nr}$ is decreasing. Therefore,
\begin{eqnarray*}\int_{0}^z \int_{\cB^-} |\la\bn(x),e_n\ra| e^{2x_nr}r^{n-1}\,ds(x)dr &\ge& \left(\int_{\cB^-} |\la\bn(x),e_n\ra| ds(x)\right)\left(\int_{0}^z e^{-2h_\Omega r}r^{n-1}\,dr\right)\\&=&\left(\int_{\cB^-} |\la\bn(x),e_n\ra| ds(x)\right)\frac{1}{(2h_\Omega)^n}\left(\Gamma(n)-\Gamma(n,2h_\Omega z)\right).
\end{eqnarray*}
Similarly 
\begin{eqnarray*}\int_{0}^z \int_{\cB^+} \la\bn(x),e_n\ra e^{2x_nr}r^{n-1}\,ds(x)dr&\le&\left(\int_{\cB^+} \la\bn(x),e_n\ra ds(x)\right)\left(\int_{0}^z e^{-2\delta r}r^{n-1}\,dr\right)\\&=&\left(\int_{\cB^+} \la\bn(x),e_n\ra ds(x)\right)\frac{1}{(2\delta)^n}\left(\Gamma(n)-\Gamma(n,2\delta z)\right).
\end{eqnarray*}
This completes the proof. 
\end{proof}
Theorem~\ref{maincor}   is an immediate consequence of Theorem~\ref{main} and Corollary~\ref{maincor2}.
\begin{proof}[\bf Proof of Theorem~\ref{maincor}] Let  $\tilde{\Omega}:=\cF\times (-h_\Omega,0)$. According to Lemma~\ref{monotoni}
\[R_1^\Omega(z)\ge R_1^{\tilde\Omega}(z).\]
Thus, it is enough to find a lower bound for $\tilde R_1(z)$. Using Corollary~\ref{maincor2} we conclude 
\[  R_1^{\tilde\Omega}(z)\ge C_{n,1}|\cF|z^{n}+\frac{(n-1)\omega_{n-1}}{(2\pi)^{n-1}}\left(\frac{|\cF|}{(2h_\Omega)^n}\left(\Gamma(n)-\Gamma(n,2h_\Omega z)\right)\right).\]
\end{proof} 
\begin{rem}\label{cylindr}For a cylindrical domain $\cF\times(-h,0)$, the sloshing eigenvalues can be calculated explicitly using separation of variable (see~\cite{Petal10}). They  are of the form
 \[\sqrt{\mu_k}\tanh({\sqrt{\mu_k}h_\Omega}),\]
 where $\mu_k$ is the $k$-th Neumann eigenvalues of the Laplacian on $\cF$. 
One can try to get an estimate for the Riesz means using this explicit expression of the eigenvalues.  We shall see below that it does not give an asymptotically sharp bound. \\For $\Omega$ we have
 \begin{eqnarray*}R_1(z)&=&\sum_{k}(z-\sqrt{\mu_k}\tanh({\sqrt{\mu_k}h}))_+\\&\ge&\sum_{k}(z-\sqrt{\mu_k})_+\\
&\ge&\frac{1}{2z}\sum_{k}(z^2-{\mu_k})_+\\
&=&\frac{1}{2z}{R}_1^\cF(z^2,\Delta_N),\end{eqnarray*}
   where ${R}^\cF_1(z^2,\Delta_N)$ is the Riesz mean of the Neumann Laplace eigenvalues on $\cF$. We can use Harrell--Stubbe's result~\cite{HS16} on lower bounds for ${R}^\cF_1(z^2,\Delta_N)$ to get
\begin{eqnarray*}
 {R}^\cF_1(z^2,\Delta_N)= \sum_{k}(z^2-\mu_k)_{+}&\ge& L_{1,n-1}^{cl}|\cF|z^{n+1}+\frac1{4}\,L_{1,n-2}^{cl}\frac{|\cF|}{\delta_{\bf v}(\cF)}\,z^{ n  }\\&&-\frac{1}{96}\,(2\pi)^{2-n}\omega_n\frac{|\cF|}{\delta_{\bf v}(\cF)^2}\,z^{ {n-1} },
\end{eqnarray*}
where $\delta_{\bf v}(\cF)$ is the width of $\cF$ in the direction of $\bf v\in\R^{n-1}$ and
\begin{equation}\label{ccnst}
  L_{1,n-1}^{cl}:=   \frac{1}{(4\pi)^{\frac{n-1}{2}}\Gamma(1+{\frac {n+1} 2})}.
\end{equation}
Comparing $ L_{1,n-1}^{cl}$ with $C_{n,1}$ we have
\begin{equation}L_{1,n-1}^{cl}=\frac{2n}{n+1}{C_{n,1}}.\end{equation}
Therefore
\begin{eqnarray*}R_1(z)&\ge& \frac{n}{n+1}C_{n,1}|\cF|z^{n}+\frac1{8}\,L_{1,n-2}^{cl}\frac{|\cF|}{\delta_{\bf v}(\cF)}\,z^{ n-1  }\\&&-\frac{1}{192}\,(2\pi)^{2-n}\omega_n\frac{|\cF|}{\delta_{\bf v}(\cF)^2}\,z^{ {n-2} }.
\end{eqnarray*}
By Lemma~\ref{lemav}, this bound holds for any proper subset $\Omega$ of $\cF\times(-h,0)$ with $\partial\Omega=\cF\cup\cB$.  It gives a two--term lower bound only depending  on the geometry of $\cF$.  When $n\to\infty$, the coefficient of the leading term tends to the optimal constant $C_{n,1}$.  
 \end{rem}
We now prove the main theorem. 
\begin{proof}[\bf Proof of Theorem~\ref{main}] The proof follows from Lemma 2.1 choosing a suitable  family of  test functions. 
Consider the family of harmonic functions 
$$ f_{\xi'}(x)=e^{ix'\xi'+x_n|\xi'|}$$ 
where $x=(x',x_n)\in \R^{n-1}\times\R$ and $\xi'\in \R^{n-1}$. Replacing in \eqref{mvarg}  with  $M=\R^{n-1}$ and $M_0=\{|\xi'|\le z\}$ we get

\begin{align*}\sum_j(z-\nu_j)_+\int_{\R^{n-1}}|\hat\varphi_j(\xi')|^2\,d\xi'\ge& |\cF|\int_{|\xi'|\le z}(z-|\xi'|)\,d\xi'\\&-\int_{|\xi'|\le z} \int_{\cB} \la\bn,(0,|\xi'|)\ra e^{2x_n|\xi'|}\,dsd\xi' , \end{align*}
where $\hat{\varphi}_j(\xi')=\int_\cF e^{ix'\xi'}\varphi_j(x')ds$ is the Fourier transform of $\varphi_j|_\cF$. 
Therefore,
\[R_1(z)=\sum_j(z-\nu_j)_+\ge\frac{\omega_{n-1}}{n(2\pi)^{n-1}}|\cF| z^n -\frac{(n-1)\omega_{n-1}}{(2\pi)^{n-1}}\int_{0}^z \int_{\cB} \la\bn,e_n\ra e^{2x_nr}r^{n-1}\,dsdr, \]
where $\omega_{n-1}$ is the volume of a unit ball in $\R^{n-1}$. 
Proceeding with the Riesz iteration, we obtain inequality~\eqref{rieszinq}. This completes the proof. \end{proof}
We end this subsection with an example. 
\begin{ex}Let consider the cone 
$$
\cC:=\{(x,y,z): \tan^2(\alpha)(x^2+y^2)=(z+h)^2,\quad z\in( -h,0)\}\subset \R^3,
$$  
where $\alpha$ is the interior angle between $\cB$ and the free surface $\cF=\cC\cap \R^2\times\{0\}$. Computing \eqref{a1gamma} we obtain
\begin{eqnarray*}A_{3,1}(z)&=&\cos(\alpha)\int_{0}^z \int_0^{\frac{h}{|\tan(\alpha)|}} \frac{1}{|\cos(\alpha)|}e^{-2|\tan(\alpha)|tr}tr^{2}\,dt dr\\&=&\frac{1}{4|\tan(\alpha)|\tan(\alpha)}\int_{0}^z\left(1-e^{-2hr}-2hre^{-2hr}\right)dr\\
&=&\frac{1}{4|\tan(\alpha)|\tan(\alpha)}\big(z-\frac{1}{2h}(1-e^{-2hz}+ze^{-2hz})+\frac{1}{4h^2}(1-e^{-2hz})\big).
\end{eqnarray*}
Hence,
\[R^{\cC}_{3,1}(z)\ge\frac{1}{12\pi}|\cF|z^{3} +\frac{1}{4|\tan(\alpha)|\tan(\alpha)}z+c,\]
where $c=A_{3,1}(z)-\frac{1}{4|\tan(\alpha)|\tan(\alpha)}z$.\\
 One can ask if we can improve the power of $z$ in the second term. \end{ex}
\subsection{Riesz means of sloshing problem on domains in $\R^2$} In this section, we prove Theorem~\ref{triangle}. Let us begin with the example which will be  used   in the proof of Theorem~\ref{triangle}.

\begin{ex}[Triangular domain]\label{tri1} Let $\Omega\subset \R^2$ be a triangle  with interior angles  $\alpha,\beta\in(0,\pi)$ as shown in Figure 2 (note that   all the following calculations remain the same if  one considers a trapezoid).      Segment $\overline{OQ}$ is the free part $\cF$ of the boundary  with length $L$, and  $\cB=\overline{OP}\cup\overline{PQ}$. Replacing in \eqref{a1gamma} we have
\begin{figure}\label{triangle1}\includegraphics{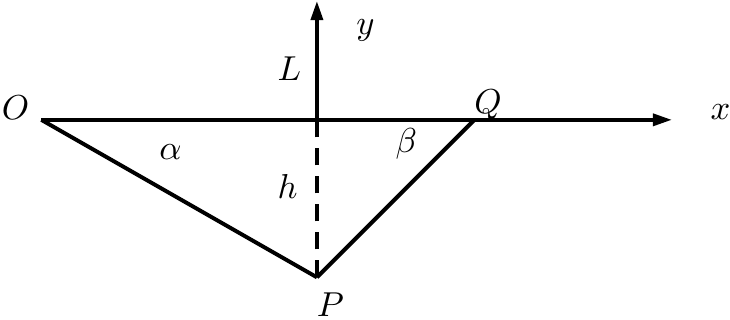}\caption{}\end{figure}
\[A_{2,1}(z)=-\frac{1}{\pi}\int_{0}^z \int_{\overline{OP}\cup\overline{PQ}} \la\bn,(0,1)\ra re^{2yr}\,dsdr .\]
First  we calculate the above integral for $\alpha,\beta\in(0,\pi)\setminus\{\frac{\pi}{2}\}$
\begin{align*}-\int_{0}^z \int_{\overline{OP}} \la\bn,(0,r)\ra re^{2yr}\,dsdr&= \int_{0}^zr\cos(\alpha)\left( \int_{0}^{\frac{h}{|\tan(\alpha)|}} e^{-2x|\tan(\alpha)|r}\frac{1}{|\cos(\alpha)|}\,dx\right)dr\\
&=\int_{0}^z\frac{1}{2\tan(\alpha)}\left(1- e^{-2hr}\right)\,dr\\
&=\frac{1}{2\tan(\alpha)}\left(z-\frac{1}{2h}(1-e^{-2hz})\right).\end{align*}
Similarly, we get 
\[-\int_{0}^z \int_{\overline{PQ}} \la\bn,(0,r)\ra e^{2yr}\,dsdr=\frac{1}{2\tan(\beta)}\left(z-\frac{1}{2h}(1-e^{-2hz})\right).\]
If $\alpha$ or $\beta$ is equal to $\frac{\pi}{2}$ then $\la\bn,(0,1)\ra=0$ on $\overline{OP}$ or $\overline{PQ}$ respectively. We make a convention that $\frac{1}{\tan(\frac{\pi}{2})}=0$.
Therefore, we get an explicit formula for $A_{2,1}(z)$ in terms of interior angles $\alpha,\beta\in(0,\pi)$:
\[A_{2,1}(z)=\frac{1}{2\pi}\left(\frac{1}{\tan(\alpha)}+\frac{1}{\tan(\beta)}\right)\left(z-\frac{1}{2h}(1-e^{-2hz})\right).\]
We conclude\begin{equation}R_1(z)\ge \frac{1}{2\pi}|\cF|z^{2}+\frac{1}{2\pi}\left(\frac{1}{\tan(\alpha)}+\frac{1}{\tan(\beta)}\right)\left(z-\frac{1}{2h}(1-e^{-2hz})\right).\end{equation}
\end{ex}\smallskip

\begin{proof}[\bf Proof of Theorem~\ref{triangle}]
%\begin{figure}\includegraphics{pic5.pdf}\caption{}\end{figure}
%
By Lemma~\ref{monotoni},  we know 
$$R^{\Omega}_{2,\gamma}(z)\ge R^{\tilde\Omega}_{2,\gamma}(z).$$
 Having Theorem~\ref{main}, it is enough to estimate
$A_{2,1}^{\tilde \Omega}(z)$.
 Consider Figure~\ref{pic1}. By assumption $\overline{PP_1}$ and $\overline{QQ_1}$ are line segments and $\tilde\cB_c=\tilde\cB\setminus(\overline{PP_1}\cup\overline{QQ_1})$.
 \[A_{2,1}^{\tilde\Omega}(z)=-\frac{1}{\pi}\int_{0}^z \int_{\overline{PP_1}\cup\overline{QQ_1}\cup{\tilde\cB_c}} \la\bn,(0,r)\ra e^{2yr}\,dsdr.\]
 Following the same calculation as in Example~\ref{tri1}  we obtain
 \[-\frac{1}{\pi}\int_{0}^z \int_{\overline{PP_1}\cup\overline{QQ_1}} \la\bn,(0,r)\ra e^{2yr}\,dsdr=\frac{1}{2\pi}\left(\frac{1}{\tan(\alpha)}+\frac{1}{\tan(\beta)}\right)\left(z-\frac{1}{2\delta}(1-e^{-2\delta z})\right).\]
 In order to compute the remaining term, let $$\tilde\cB_{c}^+:=\{(x,y)\in\tilde\cB\,:\, \la\bn,(0,1)\ra=\cos(\theta(x,y))\ge0\}$$ and $$\tilde\cB_{c}^-:=\{(x,y)\in\tilde\cB\,:\, \la\bn,(0,1)\ra=\cos(\theta(x,y))<0\},$$ 
 where $\theta(x,y)$ is the angle between $\bn$ and the $(0,1)$.
 Hence
 \begin{eqnarray*}-\frac{1}{\pi}\int_{0}^z \int_{\tilde\cB_c^+\cup \tilde\cB_c^-} \la\bn,(0,1)\ra re^{2yr}\,dsdr&\ge&- \frac{1}{\pi}\int_{0}^z \int_{\tilde\cB_c^+} r\cos(\theta(x,y))e^{-2\delta r}\,dsdr\\
\nonumber&& + \frac{1}{\pi}\int_{0}^z \int_{\cB_c^-} r|\cos(\theta(x,y))|e^{-2h r}\,dsdr\\
&=&-\frac{1}{2\pi\delta}\left(\frac{1}{2\delta}-\frac{1}{2\delta}e^{-2\delta z}-ze^{-2\delta z}\right)\int_{\tilde\cB_c^+} \cos(\theta(x,y))ds\\
\nonumber&&\frac{1}{2\pi h}\left(\frac{1}{2h}-\frac{1}{2h}e^{-2h_z}-ze^{-2h z}\right)\int_{\cB_c^-} |\cos(\theta(x,y))|ds\\
\nonumber&\ge&-\frac{1}{2\pi\delta}\left(\frac{1}{2\delta}-\frac{1}{2\delta}e^{-2\delta z}-ze^{-2\delta z}\right)|\tilde\cB_c|.
 \end{eqnarray*}
   Hence 
 \[R^{\Omega}_{2,1}(z)\ge\frac{1}{2\pi}|\cF|z^{2}+ \frac{1}{2\pi}\left(\frac{1}{\tan(\alpha)}+\frac{1}{\tan(\beta)}\right)z+c,\]
 where 
 $$c=\frac{1}{4\pi\delta}\left(\frac{1}{\tan(\alpha)}+\frac{1}{\tan(\beta)}\right)(1-e^{-2\delta z})-\frac{1}{2\pi\delta}\left(\frac{1}{2\delta}-\frac{1}{2\delta}e^{-2\delta z}-ze^{-2\delta z}\right)|\tilde\cB_c|.$$  
 Applying the Riesz iteration on both sides of the inequality completes the proof. \end{proof}
 It is clear from the proof of Theorem~\ref{triangle} that it is not necessary  to assume that $\cF$ is connected in the statement.   \\
If in Theorem~\ref{triangle} $\alpha=\beta=\frac{\pi}{2}$, then 
$A_{2,\gamma}(z)=O(z^{\gamma-1})$. Hence, the power of $z$ in the second term of lower bound \eqref{1.4} is not optimal.  However, for a rectangular domain we can do a more explicit computation of its Reisz means and get a two-term asymptotically sharp lower bound. It immediately leads to the same bound on domains satisfying the John condition. 

%

%\begin{figure}
%\centering{
%\input{circle1.pdf_tex}
%\caption{}\label{circle}}
%\end{figure}
\begin{prop}\label{correctangle} Let $\Omega$ be a bounded domain in $\R^2$ as in \eqref{hypothesis} with free part $\cF=(0,\ell)\times\{0\}$. Assume that $\Omega$ satisfies the John condition  then
 \[R^\Omega_\gamma(z)\ge\frac{\ell}{\pi(\gamma+1)}z^{\gamma+1}+\frac{1}{2}z^{\gamma}.\]

\end{prop}
\begin{proof}
By Lemma \ref{monotoni}, it is enough to prove the inequality for  $\cR=(0,\ell)\times(-h_\Omega,0)$
 where $h_\Omega$ is the depth of $\Omega$. With the notation of Remark~\ref{cylindr}, we have $\mu_k=\frac{k^2\pi^2}{\ell^2}$, $k\in\Z_+$, and 
  \[\nu_k(\cR)=\frac{k\pi}{\ell}\tanh\left(\frac{k\pi}{\ell}h_\Omega\right).\]
  Hence, we have
 \[R_1^\cR(z)=\sum_{k}(z-\nu_k)_+=\sum_{k}\left(z-\frac{k\pi}{\ell}\tanh\left(\frac{k\pi}{\ell}h_\Omega\right)\right)_+\ge\sum_{k}\left(z-\frac{k\pi}{\ell}\right)_+.\]
 We now use the following simple Lemma. 
 \begin{lem}\label{simple} For any $R\ge0$ and $k\in \Z_+$ we have
 \[\frac{1}{2}\left(R^2+R\right)\le\sum_{k\ge0}(R-k)_+\le\frac{1}{2}\left(R^2+R+1\right).\]
 \end{lem}
 \begin{proof} The statement follows by a simple calculation. 
 \begin{eqnarray*}\sum_{k\ge0}(R-k)_+&=&R+R[R]-\frac{[R]^2}{2}-\frac{[R]}{2}\\
 &=&\frac{1}{2}\left(R^2+R\right)+\frac{1}{2}(R-[R])(1-R+[R])). \end{eqnarray*} We conclude by 
 \[0\le\frac{1}{2}(R-[R])(1-R+[R]))\le\frac{1}{2}.\]
 \end{proof}
Using the above lemma, we get
 \[R_1(z)=\frac{\pi}{\ell}\sum_{k\ge0}(\frac{\ell}{\pi}z-{k})_+\ge\frac{\ell}{2\pi}z^2+\frac{1}{2}z.\]
This completes the proof.
\end{proof}

\subsection{Bound on sum of  eigenvalues} We state and prove a more general version of Theorem~\ref{kroger}. 
\begin{thm}\label{generalthm} For $n\ge2$ the eigenvalues of problem \eqref{slosh} satisfies 
\begin{eqnarray*}\frac{1}{k}\sum_{j=1}^k\nu_j&\le&\frac{n-1}{n} \left(W_{n,k}-\frac{1}{W_{n,k}}(\nu_{k+1}-{W_{n,k}})^2\right)+W_{n,k}^{-(n-1)}c_{\cB}(\nu_{k+1}),\end{eqnarray*}
where $$c_\cB(R):=(n-1)\omega_{n-1}|\cF|^{-1}\int_0^R\int_\cB\la\bn,(0,1)\ra r^{n-1} e^{2x_nr}ds\,dr$$ and 
$W_{n,k}= 2\pi\omega_{n-1}^{-\frac{1}{n-1}}\left(\frac{k}{|\cF|}\right)^{\frac{1}{n-1}}$. 
\end{thm}
\begin{proof} The proof follows the same lines of  argument as in~\cite{HS16}. For the sake of completeness we present the whole argument. \\
By taking  $z=\nu_{k+1}$, $M=\R^{n-1}$ and $M_0= B_R$ in  Lemma~\ref{lemav}, where $B_R$ is the ball of radius $R$ centered at origin, we obtain
\begin{equation}\label{inq2}\nu_{k+1}R^{n-1}-\frac{n-1}{n}R^n\le W_{n,k}^{n-1}\left(\nu_{k+1}-\frac{1}{k}\sum_{j=1}^k\nu_j\right)+c_{\cB}(R).\end{equation}
%\begin{equation}\label{inq1}\nu_{k+1}\left(R^{n-1}-c_nk\right)\le\frac{n-1}{n}R^n-c_n\sum_{j=1}^k\nu_j+c_{\cB}\end{equation}
 \begin{rem}One can  immediately get a counterpart of Kr\"oger's inequality for the eigenvalues of the sloshing problem by setting  $R^{n-1}=c_n(k+1)$ in inequality \eqref{inq2}:
 \[\sum_{j=1}^{k+1}\nu_j\le \frac{n-1}{n}c_n^{-\frac{1}{n-1}}(k+1)^{\frac{n}{n-1}}+(2\pi)^{1-n}\int_{B_R}\int_\cB\la\bn,(0,1)\ra|\xi'| e^{2x_n|\xi'|}dsd\xi'.\]
 \end{rem}
To simplify \eqref{inq2}, take $R=W_{n,k}x$. Then
\begin{eqnarray*}\frac{1}{k}\sum_{j=1}^k\nu_j-\frac{n-1}{n}W_{n,k}&\le& \frac{n-1}{n}W_{n,k}\left(x^n-\frac{n}{n-1} \frac{\nu_{k+1}}{W_{n,k}}x^{n-1}+\frac{n}{n-1} \frac{\nu_{k+1}}{W_{n,k}}-1\right)\\&&+W_{n,k}^{-(n-1)}c_{\cB}(W_{n,k}x).\end{eqnarray*}
If now we take $x=\frac{\nu_{k+1}}{W_{n,k}}$, then
\begin{eqnarray*}\frac{1}{k}\sum_{j=1}^k\nu_j-\frac{n-1}{n}W_{n,k}&\le& \frac{1}{n}W_{n,k}\left(nx-(n-1)-x^{n}\right)+W_{n,k}^{-(n-1)}c_{\cB}(W_{n,k}x).\end{eqnarray*}
Using  the refinement of the Young inequalities stated in~\cite[Appendix A]{HS16}, for any $n\ge2$ and every $x>0$, the following inequality holds
$$nx-(n-1)-x^{n}\le-(n-1)(x-1)^2.$$
Thus we conclude 
\begin{eqnarray*}\frac{1}{k}\sum_{j=1}^k\nu_j-\frac{n-1}{n}W_{n,k}&\le& -\frac{n-1}{n}\frac{1}{W_{n,k}}(\nu_{k+1}-{W_{n,k}})^2+W_{n,k}^{-(n-1)}c_{\cB}(\nu_{k+1}).\end{eqnarray*}
\end{proof}
 
\begin{proof}[\bf Proof of Theorem~\ref{kroger}] Under the assumption of Theorem~\ref{kroger} we have
 $$\cB^-:=\{x\in\cB\,:\, \la\bn,(0,r)\ra\le0\}=\cB.$$ 
 Hence, $c_\cB(x)\le0$ for any $x>0$, and the statement follows from Theorem~\ref{generalthm}. 

\end{proof}

\section{Mixed Steklov-Dirichlet eigenvalue problem}\label{4}
In this section we  prove the results stated in  Section~\ref{intro2}. We begin by recalling the monotonicity property of the eigenvalues of Steklov--Dirichlet eigenvalue problem \eqref{dirstek}.
\begin{lem}{\cite[Proposition 3.1.1]{Petal10}}\label{monotonidir}
Let $\Omega$ and $\tilde\Omega$ be bounded domains of $\R^n$ whose boundaries $\partial\Omega=\cF\cup\cB$ and $\partial\tilde\Omega=\tilde\cF\cup\tilde\cB$ are as described in \eqref{hypothesis}.  Let $\Omega$ be a proper subset of $\tilde\Omega$ and $\tilde\cF=\cF$. %We denote eigenvalues  of problem \eqref{slosh} on $\Omega$ and $\tilde\Omega$ by $\nu_k$ and $\tilde\nu_k$ respectively. 
Then the following inequality holds
\[\eta_{k}(\Omega)>\eta_{k}(\tilde\Omega),\qquad\forall k\ge1.\]
In particular,
$$R^\Omega_\gamma(z,\cD_D)=\sum_j(z-\eta_j(\Omega))_+\le\sum_j(z-\eta_j(\tilde\Omega))_+= R^{\tilde\Omega}_\gamma(z,\cD_D).$$
\end{lem} 
\begin{proof}[\bf Proof of Theorem~\ref{adirstek}] Since $\Omega$ is subset of  the cylinder $\cF\times[-h_\Omega,0]=:\mathcal{R}$, where $h_\Omega$ is the depth of $\Omega$, by Lemma~\ref{monotonidir}  we have
$$R^\Omega_\gamma(z)\le R^\cR_\gamma(z).$$
For $\mathcal{R}$ the eigenvalues and eigenfunctions of problem \eqref{dirstek} can be explicitly calculated (see~\cite{Petal10}). They are of the form
 \[\sqrt{\lambda_k}\coth({\sqrt{\lambda_k}h_\Omega}),\]
 where $\lambda_k$ is the $k$-th Dirichlet eigenvalues of the Laplacian on $\cF$. 
 Hence,  we have
 \begin{eqnarray*}R^{\cR}_1(z)&=&\sum_{j}(z-\eta_j)_+=\sum_{j}(z-\sqrt{\lambda_j}\coth({\sqrt{\lambda_j}h_\Omega}))_+\\&\le&\sum_{j}(z-\sqrt{\lambda_j})_+.\end{eqnarray*}
 The sequence $\sqrt{\lambda_j}$ of square root of eigenvalues of  the Dirichlet Laplacian $-\Delta_D$ on $\cF$ is equal to the eigenvalues of the Navier fractional Laplacian $(-\Delta)^{1/2}_N$ on $\cF$.  We denote its $j$-th eigenvalue by $ \lambda_j((-\Delta)^{1/2}_N)$.
Musina and Nasarov in~\cite{MN14,MN16,MN16b} studied the fractional Laplacian with Navier and Dirichlet type boundary conditions. Let us recall that for an arbitrary $s>0$ the fractional Laplacian with Navier boundary conditions is defined by 
$$
(-\Delta_{Nv}^s f, f) = \sum_j \lambda_j^s |(f,\psi_j)|^2,
$$
where $\psi_j$ are eigenfunctions of the Dirichlet Laplacian $(-\Delta)_D$. The ÒDirichletÓ fractional Laplacian is defined by the closure from the class of functions  $f\in C_0^\infty(\mathcal F)$ of the quadratic form 
$$
((-\Delta)_D^s f,f) = \frac{1}{2\pi}\, \int |\xi|^{2s} | \hat f(\xi)|^{2}\, d\xi,
$$
where $\hat f$ is the Fourier transform of $f$.
It was proved in~\cite[Corollary 4]{MN14}  that for any $0<s<1$, the $j$-th eigenvalue $\lambda_j((-\Delta)^{s}_D)$ of the Dirichlet fractional Laplacian is strictly smaller than $\lambda_j((-\Delta)^{s}_{Nv})$ which implies
 \[\lambda_j((-\Delta)^{1/2}_{Nv})>\lambda_j((-\Delta)^{1/2}_D).\]
Therefore 
 \[\sum_{j}(z-\sqrt{\lambda_j})_+=\sum_{j}\left(z-\lambda_j((-\Delta)^{1/2}_{Nv})\right)_+<\sum_{j}\left(z-\lambda_j((-\Delta)^{1/2}_D)\right)_+.\]
 We now use the bound on the Riesz means of the Dirichlet fractional Laplacian $(-\Delta)^{1/2}_D$ proved by Laptev in~\cite[Corollary 2.3]{Lap97}.
 \[\sum_{j}\left(z-\lambda_j((-\Delta)^{1/2}_D)\right)_+\le (2\pi)^{-(n-1)}|\cF|z^n\left(\int_{\R^{n-1}}(1-|\xi|)_+d\xi\right)=C_{n,1}|\cF|z^n.\]
 This completes the proof. 
 \end{proof}
 \begin{rem}
 Applying the Laplace transform on inequality~\eqref{1.6}, we can get an immediate upper bound for the trace of the heat kernel of operator $\cD_D$. More precisely
 \[\sum_{j=0}^\infty e^{-\eta_jt}\le\frac{\Gamma(n)}{(4\pi)^{\frac{n-1}{2}}\Gamma(\frac{n+1}{2})}\frac{|\cF|}{t^{n-1}}.\]
 \end{rem}
 \begin{proof}[\bf Proof of Theorem~\ref{recdir}]
We first show that inequality \eqref{eqrecdir} holds for a rectangular domain   $\cR=(0,\ell)\times(-h,0)$. With the same notations as in the proof of Theorem~\ref{adirstek}, we have $\lambda_j((0,\ell))=\frac{j^2\pi^2}{\ell^2}$, $j\in\Z_+$ and 
  \[\eta_j(\cR)=\frac{j\pi}{\ell}\coth\left(\frac{j\pi}{\ell}h\right).\]
Thus 
 \[R^{\cR}_1(z)=\sum_{j}(z-\eta_j)_+=\sum_{j>0}\left(z-\frac{j\pi}{\ell}\coth\left(\frac{j\pi}{\ell}h\right)\right)_+\le\sum_{j}\left(z-\frac{j\pi}{\ell}\right)_+.\]
By Lemma~\ref{simple} we have
 \[\frac{\pi}{\ell}\sum_{j>0}\left(\frac{\ell}{\pi}z-{j}\right)_+\le\frac{\pi}{2\ell}\left(\frac{\ell^2}{\pi^2}z^2-\frac{\ell}{\pi}z+1
\right)=\frac{\ell}{2\pi}z^2-\frac{1}{2}z+\frac{\pi}{2\ell}.\]
By monotonicity property in Lemma~\ref{monotonidir}  the statement of the theorem follows. 
\end{proof}
\begin{rem}\label{lowdir}Let   $\Omega\subset \R^2_-$ be a domain with $\partial\Omega=\cF\cup\cB$, $\cF=(0,\ell)$, containing a rectangular domain $\cR=(0,\ell)\times(-h,0)$. We can  get a two--term asymptotically sharp lower bound on $R^\Omega_1(z)$ with an optimal leading term. Indeed, using  the monotonicity result together with Lemma~\ref{simple} and the following inequality   $$x\coth(x)\le1+x$$
 for $z\ge1$ we obtain
\begin{eqnarray*}R^\Omega_1(z)\ge R^\cR_1(z)&=&\sum_{j>0}\left(z-\frac{j\pi}{\ell}\coth\left(\frac{j\pi}{\ell}h\right)\right)_+\\&\ge&\sum_{j>0}\left(z-1-\frac{j\pi}{\ell}\right)_+\\&\ge&\frac{\ell}{2\pi}(z-1)^2-\frac{1}{2}(z-1)\\
&=&\frac{\ell}{2\pi}z^2-\left({\frac 1 2}+\frac{\ell}{\pi}\right)z+\frac{1}{2}. \end{eqnarray*}
% \[\eta_k=\sqrt{\mu_k}\tanh{\sqrt{\mu_k}h_\Omega}\]
\end{rem}
\section{Appendix: Two-term asymptotics for $R_\gamma(z)$\\~\\ \normalfont{by Francesco Ferrulli and Jean Lagac\'e}}\label{appendix}

In this appendix we obtain in dimension $2$, under conditions slightly different than those of Theorem \ref{triangle} two-term asymptotics \eqref{FL} and \eqref{FL2} rather than lower and upper bounds for respectively the sloshing and the Steklov-Dirichlet problem. The conditions are those required for Theorems 1.2.2 and 1.3.2, and Propositions 1.2.6 and 1.3.5 of \cite{LPPS}.

Let us start by introducing \emph{local John's condition}.
\begin{defn}
 A corner point $V$ between $\CF$ and $\CW$ is said to satisfy  \emph{local John's condition} if there exists a neighbourhood $\mathcal O_V$ of $V$ such that $\mathcal O_V \cap \CW \subset \CF \times (-\infty,0)$.
\end{defn}

We now prove the following theorem.

\begin{thm} \label{thm:riesz1}
 Let $\Omega$ be a simply connected bounded Lipschitz planar domain with the sloshing surface $\CF$ of length $L$ and walls $\CW$ which are $C^1$-regular near the corner points $A$ and $B$. Let $\alpha$ and $\beta$ be the interior angles between $\CW$ and $\CF$ at the points $A$ and $B$ resp., and assume either that
 \begin{itemize}
  \item $0 < \beta \le \alpha < \pi/2$; or
  \item $0 < \beta < \alpha = \pi/2$ and $A$ satisfies  local John's condition; or
  \item $\beta = \alpha = \pi/2$ and both $A$ and $B$ satisfy local John's condition.
 \end{itemize}
Then, the following two-term asymptotics for the Riesz mean of order $\gamma > 0$ holds as $z \to \infty$ :
 \begin{equation} \label{eq:riesz1SN}
  R_\gamma^\Omega(z,\CD_N) = C_{2,\gamma} L z^{\gamma + 1} + \frac{\pi}{8} \left(\frac 1 \alpha + \frac 1 \beta \right)z^{\gamma} + \smallo{z^{\gamma}},
 \end{equation}
 and
 \begin{equation} \label{eq:riesz1SD}
  R_\gamma^\Omega(z,\CD_D) = C_{2,\gamma} z^{\gamma + 1}  - \frac{\pi}{8} \left(\frac 1 \alpha + \frac 1 \beta \right)z^{\gamma} + \smallo{z^{\gamma}},
 \end{equation}
where $C_{2,\gamma} = ((1 + \gamma) \pi)^{-1}$.
\end{thm}

\begin{proof}
 In order to be able to prove this statement in one proof for both problems at the same time, we will make the following notational convention : we write $\nu_k^+ := \eta_k$ for the eigenvalues of the Steklov-Dirichlet problem and $\nu_k^-:= \nu_k$ for those of the sloshing problem. Similarly, we write $\CD_+ := \CD_D$ and $\CD_- :=\CD_N$.
The conditions we are assuming are exactly those of \cite{LPPS} that yield the following asymptotics for $\nu_k^\pm$ :
 \begin{equation}
 \begin{aligned}
  \nu_k^\pm &= \frac{\pi}{L} \left(k - \frac 1 2 \right) \pm \frac{\pi^2}{8 L} \left(\frac1 \alpha + \frac 1 \beta\right) + r(k) \\
  &= \frac{\pi}{L} k + C_{\pm} + r(k).
  \end{aligned}
 \end{equation}
Moreover, $r(k) = \smallo 1$. This allows us to write
\begin{equation} \label{eq:riesz}
 R_\gamma^\Omega(z;\CD_{\pm}) = \sum_{0 \le \nu_k^\pm \le z} \left(z - \frac\pi L k - C_\pm - r(k)\right)^\gamma. 
\end{equation}
Observe that $\nu_k^\pm \le z$ if and only if
\begin{equation}
 k + r(k) \le \frac{L}{\pi}\left( z - C_\pm \right).
\end{equation}
Since $r(k) = \smallo 1$, there exists a function $s(z) = \smallo 1$ such that $\nu_k^\pm \le z$ if and only if $k \le g(z)$, for $g(z) := \frac{L}{\pi}(z - C_\pm) + s(z)$. Since we start counting eigenvalues at $k = 1$, this allows us to rewrite equation \eqref{eq:riesz} as
\begin{equation}
 R_\gamma^\Omega(z;\CD_\pm) = \sum_{1 \le k \le g(z)} \left(\frac{\pi}{L}\right)^\gamma\left(g(z) -  k - s(z) - \tilde r(k)\right)^\gamma,
\end{equation}
where $\tilde r = L \pi^{-1} r$. From now on we make use of the strategy of the proof by Lagac\'e and Parnovski for \cite[Theorem 1.6]{LP}. Let us write $g(z) = a_z + \tau_z$ where $a_z$ is the integer part and $\tau_z$ the fractional part, and rewrite the previous sum as
\begin{equation}
  R_\gamma^\Omega(z;\CD_\pm) =\left(\frac \pi L\right)^\gamma \sum_{0 \le k \le a_z - 1}  (k + \tau_z - s(z) - \tilde r(k))^{\gamma}. 
\end{equation}
Since $\tilde r(k) = \smallo 1$, there exists some $K$ such that for all $k > K$, $\tilde r(k) < 1/4$. We split the sum into
\begin{equation}
\sum_{0 \le k \le a_z - 1}  (k + \tau_z - s(z) - \tilde r(k))^{\gamma}  = \left(\sum_{0 \le k \le K} + \sum_{K < k \le a_z - 1}\right) (k + \tau_z - s(z) - \tilde r(k))^{\gamma}.
\end{equation}

We have that
\begin{equation}
 \begin{aligned}
  \sum_{0 \le k \le K} \left(k + \tau_z - s(z) - \tilde r(k) \right)^\gamma &\le K\left(K + 1 + \smallo z + \inf_{k \le K} r(k)\right)^\gamma \\
  &= \smallo{z^{\gamma}}.
 \end{aligned}
\end{equation}

For the second sum, suppose that we have chosen $z$ large enough that $s(z) < K/4$. Then, we have that
\begin{equation}
 \begin{aligned}
  \sum_{K < k \le a_z - 1} \left(k + \tau_z - s(z) - \tilde r(k) \right)^\gamma &= 
  \sum_{K < k \le a_z - 1} (k + \tau_z)^{\gamma}\left(1 - \frac{s(z) - \tilde r(k)}{k + \tau_z} \right)^\gamma \\
  & =  \sum_{K < k \le a_z - 1}\left( (k + \tau_z)^{\gamma}  - s(z) \: \bigo{k^{\gamma - 1}} + \smallo{k^{\gamma - 1}}\right), \\
  &= \smallo{z^{\gamma}} + \sum_{K < k \le a_z - 1} (k + \tau_z)^{\gamma}
 \end{aligned}
\end{equation}

Finally, the Euler-Maclaurin formula tells us that
\begin{equation}
\begin{aligned}
  \sum_{K < k \le a_z - 1} \left(k + \tau_z\right)^\gamma  &= \int_{K+1}^{a_z-1} (k + \tau_z)^\gamma \de k + \frac{1}{2} \left((K+1 + \tau_z)^\gamma + (g(z) - 1)^\gamma\right) + \bigo{z^{\gamma - 1}} \\
 &= \int_{K + 1 + \tau_z}^{g(z) - 1} k^{\gamma} \de k + \frac{1}{2} \left((K+1 + \tau_z)^\gamma + (g(z) - 1)^\gamma\right) + \bigo{z^{\gamma - 1}} \\
 &= \frac{1}{\gamma + 1} (g(z) - 1)^{\gamma + 1} + (g(z) - 1)^{\gamma} + O(1) + \bigo{z^{\gamma - 1}}
 \end{aligned}
\end{equation}
We reexpand $g(z)$ and $C_\pm$ and collect all terms together to obtain directly that
\begin{equation}
 R_\gamma^\Omega(z;\CD_N) =  C_{2,\gamma}L z^{\gamma + 1} + \frac{\pi}{8} \left(\frac 1 \alpha + \frac 1 \beta \right)z^\gamma + \smallo{z^{\gamma}},
\end{equation}
and
\begin{equation}
  R_\gamma^\Omega(z;\CD_D) = C_{2,\gamma}L z^{\gamma +1} - \frac{\pi}{8} \left(\frac 1 \alpha + \frac 1 \beta \right)z^{\gamma} + \smallo{z^\gamma}
\end{equation}
finishing the proof.
\end{proof}

%\bibliography{bibsloshing}
%\bibliographystyle{plain}

\end{document}